\documentclass[11pt]{article}
\usepackage[a4paper,margin=2.8cm]{geometry}
\usepackage{amsmath,amssymb,amsthm,mathtools}
\usepackage{enumitem,microtype,hyperref,mathrsfs,tabularx}
\theoremstyle{plain}
\newtheorem{theorem}{Theorem}[section]
\newtheorem{proposition}[theorem]{Proposition}
\newtheorem{lemma}[theorem]{Lemma}
\newtheorem{claim}{Claim}
\newtheorem{corollary}[theorem]{Corollary}
\theoremstyle{definition}
\newtheorem{remark}[theorem]{Remark}

\newcommand{\PP}{\mathbb P}
\newcommand{\QQ}{\mathbb Q}
\newcommand{\ZZ}{\mathbb Z}\newcommand{\CC}{\mathbb C}
\newcommand{\OO}{\mathcal O}\newcommand{\F}{\mathcal F}
\newcommand{\Sing}{\operatorname{Sing}}
\DeclareMathOperator{\ord}{ord}\newcommand{\Gr}{\operatorname{Gr}}

\title{An arithmetic integrability result for codimension-one foliations on complex projective spaces}
\author{V\'ictor Le\'on
\thanks{ILACVN--CICN, Universidade Federal da Integração Latino-Americana,
Parque Tecnológico Itaipu, Foz do Iguaçu--PR, 85867-970, Brazil.
E-mail: \href{mailto:victor.leon@unila.edu.br}{victor.leon@unila.edu.br}}
\and Bruno Sc\'ardua
\thanks{Instituto de Matem\'atica, Universidade Federal do Rio de Janeiro,
CP 68530, Rio de Janeiro--RJ, 21945-970, Brazil.
E-mail: \href{mailto:bruno.scardua@gmail.com}{bruno.scardua@gmail.com}}
}
\date{}
\begin{document}
\maketitle

\begin{abstract}
Let $\F$ be a codimension-one holomorphic foliation of degree $d$ on
$\PP^n$, $n\geq3$, admitting an invariant hyperplane $H$. We study the
extremal situation in which
$S=(H\cap\Sing(\F))_{\rm red}$ is an irreducible hypersurface of $H$
of degree $d+1$. When $d+1$ is a power of a prime, we prove that, in
suitable homogeneous coordinates with $H=(t=0)$,
\[
 \Omega=Q\,dt-\frac{t}{d+1}\,dQ ,
\]
where $Q$ is homogeneous of degree $d+1$. Thus $Q/t^{d+1}$ is a
rational first integral. The proof reduces the Frobenius equation to a
twisted closedness equation on a plane section and uses Zariski's
theorem on the Alexander polynomial of an irreducible plane curve.
We also prove a complementary rigidity theorem for an arbitrary smooth
invariant hypersurface $D\subset\PP^n$: if the reduced singular divisor
on $D$ is smooth, irreducible, and of maximal degree, then the same
normal-form phenomenon holds, with no arithmetic hypothesis on its
degree; in the low-weight range the smoothness assumption on the
singular divisor can be dropped. Finally, we show that the principal
hypotheses are sharp. Dropping the maximal-degree condition yields a
family with irreducible reduced singular support and no non-constant
rational first integral. For every $d+1$ which is not a prime power we
construct a global counterexample with irreducible maximal-degree
singular support, and a final family shows that irreducibility of the
reduced support is also genuinely necessary.
\end{abstract}
\medskip
\noindent\textbf{Keywords.} Holomorphic foliations; invariant hyperplanes;
rational first integrals; Milnor fibers; Alexander polynomial; Kupka
components; cyclic covers.

\smallskip
\noindent\textbf{2020 Mathematics Subject Classification.}
32S65, 37F75, 32S55, 14F35.

\section{Introduction}

This work explores the relation between arithmetics and integrable structures in the complex algebraic framework. In \cite{CerveauLoray} the authors prove the type of Frobenius integrability theorem for singular codimension one holomorphic foliation germs in dimension $\geq 3$. They assume that the tangent cone is irreducible and has degree $p^s$, i.e., a prime-power. This work is dedicated to the search of global versions of this local result. For this we consider  a codimension one (algebraic) holomorphic foliation on the complex projective space $\mathbb P^n, \, n \geq 3$ and assume that the foliation admits an invariant hypersurface where the singular set of the foliation has degree 
a power of a prime. The first question is  whether an invariant
hyperplane $H$ for which the reduced support of
$H\cap\Sing(\F)$ is irreducible of prime-power degree, together with a
non-trivial transverse holonomy condition, forces rational
integrability. We show that this is false without a {\em maximal-degree condition} on the
reduced singular support. On the other hand, the case of maximal degree of the reduced singular
support is rigid.

\begin{theorem}[Rigidity theorem]\label{thm:main}
Let $\F$ be a codimension-one holomorphic foliation of degree $d$ on
$\PP^n$, $n\geq3$, admitting an invariant hyperplane $H$. Suppose that the reduced hypersurface
\(
 S=\bigl(H\cap\Sing(\F)\bigr)_{\mathrm{red}}
\)
is irreducible and
\(
 \deg S=d+1=p^s
\)
for a prime $p$ and $s\geq1$. Then there are homogeneous coordinates
$[t:x_1:\cdots:x_n]$, with $H=(t=0)$, and a homogeneous polynomial
$Q$ of degree $d+1$ such that $\F$ is represented by
\(
 \Omega=Q\,dt-\frac{t}{d+1}\,dQ.
\)
In particular, $Q/t^{d+1}$ is a rational first integral.
\end{theorem}

A noteworthy feature of the proof is that the prime-power condition $d+1=p^s$,
although historically inherited from the Cerveau--Loray theorem \cite{CerveauLoray},
reappears independently and intrinsically in the global problem.
This time it is required in some cohomology computations. 

The arithmetic hypothesis disappears when the singular divisor on the
invariant hypersurface is smooth.  The following  result
will be proved in Section~\ref{sec:smooth-hypersurface}.

\begin{theorem}[Smooth invariant hypersurface]\label{thm:smooth-hypersurface}
Let $\F$ be a codimension-one holomorphic foliation of degree $d$ on
$\PP^n$, $n\geq3$, and let
\(
 D=(F=0)\subset\PP^n
\)
be a smooth invariant hypersurface of degree $e$. Put
\(
 r=d-e+2.
\)
Suppose that
\(
 S=\bigl(D\cap\Sing(\F)\bigr)_{\rm red}
\)
is a smooth irreducible hypersurface of $D$ and is maximal
 \[
 \deg_{\PP^n}S=er.
\]
Then there is a homogeneous polynomial $Q$ of degree $r$ such
that $\F$ is represented by
\[
 \Omega=Q\,dF-\frac er F\,dQ.
\]
In particular,
\[
 \frac{Q^e}{F^r}
\]
is a rational first integral of $\F$.
\end{theorem}

Just one qualification about the maximality hypothesis above: 
the maximality hypothesis 
\(\deg_{\PP^n}S=er\) 
is equivalent to 
\(
 S\in |\mathcal O_D(r)|.
\)
Here $|\mathcal O_D(r)|$ denotes the complete linear system of divisors
on $D$ cut out by sections of $\mathcal O_D(r)$, equivalently by
hypersurfaces of degree $r$ in $\PP^n$. In the present setting the
intersection $D\cap\Sing(\F)$ is cut on $D$ by a section of
$\mathcal O_D(r)$; hence its divisor has projective degree $er$.
Since $S$ is the reduced support of that divisor, the maximality
condition above is exactly the multiplicity-one case, and is therefore
equivalently expressed by
\[
 \deg_{\PP^n}S=er.
\]

  When
$e=1$, Theorem~\ref{thm:smooth-hypersurface} specializes to the
hyperplane situation.  On the overlap where the singular divisor $S$
is smooth, it is stronger than Theorem~\ref{thm:main}, since no
prime-power assumption is required.  On the other hand,
Theorem~\ref{thm:main}  allows $S$ to be
irreducible and singular, provided that
$\deg S=d+1$ is a prime power.

In the low-weight range we obtain:

\begin{corollary}[Quadric and cubic cases]\label{cor:quadric-cubic}
Let $\F$ be a codimension-one holomorphic foliation of degree $d$ on
$\PP^n$, $n\geq3$, and let
\(
 D=(F=0)\subset\PP^n
\)
be a smooth invariant hypersurface of degree $e$. Put
\(
 r=d-e+2.
\)
Assume that
\(
 S=\bigl(D\cap\Sing(\F)\bigr)_{\rm red}
\)
is reduced and irreducible, not necessarily smooth, and is maximal, 
\(
 \deg_{\PP^n}S=er.
\)
If
\(
 0<r\le e,
\)
then there is a homogeneous polynomial $Q$ of degree $r$ such that
\(
 \Omega=Q\,dF-\frac er F\,dQ,
\)
and $Q^e/F^r$ is a rational first integral of $\F$.
In particular, in every dimension $n\ge3$, this applies to a smooth
invariant quadric for $d=1,2$, and to a smooth invariant cubic
hypersurface for $d=2,3,4$.
\end{corollary}

Our results are geometrically  related to the classical theory of codimension-one
foliations with complete-intersection Kupka components, where rational
first integrals arise under suitable geometric hypotheses; see
Cerveau--Lins Neto \cite{CLNKupka}. 
Theorem~\ref{thm:main} though  is formulated
 without any smoothness assumption on the reduced singular support $S\cap H$ 
In particular, the theorem applies when the irreducible support $S$ is
singular. In Section~\ref{sec:sharpness} we give examples showing that our hypotheses  cannot be dropped. 
\section{Homogeneous 1-forms and projective foliations}
\label{section:homogeneous}
A codimension one holomorphic foliation with singularities on a complex manifold $M$ consists of a pair $\mathcal F= (\mathcal F_0, S)$ where $S\subset M$ is a codimension $\geq 2$ analytic subset of $M$ and $\mathcal F_0$ is a (nonsingular) codimension one holomorphic foliation on the open manifold $M\setminus S$. We may assume that $S\subset M$ is minimal in the sense that $\mathcal F_0$ does not extend as a nonsingular holomorphic foliation to any point $p \in S$. We shall call $S$ the singular set of $\mathcal F$ and write $\Sing(\mathcal F)=S$. In a complex projective space $\mathbb P^n, n \geq 3$ such foliations admit the following description (\cite{LinsNetoScardua}). 

Put $m=d+1$.  In homogeneous coordinates
$z=(z_0,\ldots,z_n)$, a degree-$d$ codimension-one foliation on
$\PP^n$ is represented by a homogeneous polynomial $1$-form
\[
 \Omega=\sum_{j=0}^n A_j(z)\,dz_j,
 \qquad \deg A_j=m=d+1.
\]
The 1-form $\Omega$ must satisfy the following conditions. 
If
\[
 R=\sum_{j=0}^n z_j\frac{\partial}{\partial z_j}
\]
denotes the Euler (radial) vector field, the projective condition is
the horizontality identity
\[
 i_R\Omega=\sum_{j=0}^n z_jA_j=0.
\]
The Frobenius integrability condition is
\[
 \Omega\wedge d\Omega=0,
\]
and we take $\Omega$ {\it primitive} or {\it saturated}, meaning that
\[
 \gcd(A_0,\ldots,A_n)=1.
\]
More generally, throughout the paper a homogeneous polynomial
$1$-form $\alpha$ will be  called \emph{horizontal} if
\[
 i_R\alpha=0.
\]
Equivalently, $\alpha_z(z)=0$ at every point $z$.
 
Part of our strategy is based on the following. 
Choose coordinates $[t:X]=[t:x_1:\cdots:x_n]$ with $H=(t=0)$. Since
$H$ is invariant,
\begin{equation}\label{eq:decomp}
 \Omega=A(t,X)\,dt+t\,\beta(t,X),
\end{equation}
where $\beta$ is a polynomial $1$-form involving only the $dX$-directions.
If $R_H=\sum x_i\partial_{x_i}$, the Euler condition gives
\begin{equation}\label{eq:euler}
 A=-i_{R_H}\beta.
\end{equation}
Let
\[
 S_{\mathrm{red}}=\bigl(H\cap\Sing(\F)\bigr)_{\mathrm{red}}
\]
denote the reduced analytic hypersurface underlying the intersection.
Under the hypotheses of Theorem~\ref{thm:main}, $S_{\mathrm{red}}$ is
irreducible of degree $m$. Since $A(0,X)$ is homogeneous of degree
$m$ and its reduced zero set is precisely $S_{\mathrm{red}}=(P=0)$,
with $P$ reduced irreducible of degree $m$, we have
\[
 A(0,X)=cP(X)
\]
for some $c\in\CC^*$. After multiplying $\Omega$ by a non-zero
constant, we may assume $A(0,X)=P(X)$. Thus in the extremal case
$\deg S_{\mathrm{red}}=m=d+1= p^s$ the transverse multiplicity along
$S_{\mathrm{red}}$ is automatically one.

Expand
\[
 \beta=\sum_{j=0}^{m-1}t^j\beta_j,\qquad
 A=\sum_{j=0}^{m-1}t^jA_j.
\]
Then $A_j=-i_{R_H}\beta_j, \, j=0,...,m-1$. Define
\[
 Q_j=\frac{m}{m-j}A_j,\qquad
 Q=\sum_{j=0}^{m-1}t^jQ_j,
\]
and
\begin{equation}\label{eq:model}
 \widehat\Omega=Q\,dt-\frac{t}{m}dQ.
\end{equation}
The $dt$-coefficients of $\Omega$ and $\widehat\Omega$ coincide and
\[
 \widehat\Omega=-\frac{t^{m+1}}m
 d\left(\frac{Q}{t^m}\right).
\]
Set
\[
 \alpha_j=\beta_j+\frac{1}{m-j}d_HA_j.
\]
Then
\begin{equation}\label{eq:difference}
 i_{R_H}\alpha_j=0,\qquad
 \Omega-\widehat\Omega=\sum_{j=0}^{m-1}t^{j+1}\alpha_j.
\end{equation}
Using these computations we obtain:
\begin{lemma}[First-deviation equation]\label{lem:first}
Let  $\alpha_j$ be the first non-zero term in \eqref{eq:difference},
then
\begin{equation}\label{eq:twisted}
 P\,d\alpha_j+\frac{m-j}{m}\alpha_j\wedge dP=0.
\end{equation}
\end{lemma}

\begin{proof}
Write
\[
 \Omega=\widehat\Omega+t^{j+1}\alpha_j+O(t^{j+2}).
\]
Since $\widehat\Omega$ is integrable,
\[
 d(t^{j+1}\alpha_j)
 =(j+1)t^jdt\wedge\alpha_j+t^{j+1}d_H\alpha_j.
\]
Using
\[
 \widehat\Omega=A(t,X)\,dt-\frac{t}{m}dP+O_H(t^2),\qquad
 d\widehat\Omega=-\frac{m+1}{m}dt\wedge dP+O(t),
\]
where $O_H(t^2)$ denotes a $1$-form involving only the $dX$-directions
and divisible by $t^2$, the coefficient of $t^{j+1}$ in
$\Omega\wedge d\Omega$ is
\[
 dt\wedge\left(
 P\,d\alpha_j+\frac{m-j}{m}\alpha_j\wedge dP\right).
\]
The expression in parentheses is a $2$-form involving only the
$dX$-directions. Since wedging with $dt$ is injective on such forms, the
expression in parentheses itself vanishes.
\end{proof}

\section{Twisted vanishing on a plane}
\label{subsection:zariskivanishing}
We now isolate the cohomological vanishing principle that underlies the proof of Theorem~\ref{thm:main}, starting with the case of an irreducible plane curve.
For this we shall first recall some classical notions and some specific facts from Complex Algebraic Geometry. 
\subsection{Milnor fiber, monodromy, and the cohomological vanishing}

We recall the standard facts that will be used in the proof below.
Let $P\in\CC[x,y,z]$ be a reduced irreducible homogeneous polynomial
of degree $m$, and let
\[
 C=(P=0)\subset\PP^2.
\]
Its global Milnor fiber is
\[
 F=P^{-1}(1)=\{P=1\}\subset\CC^3.
\]
It is smooth: if $R=x\partial_x+y\partial_y+z\partial_z$ is the Euler
field, then Euler's identity gives
\[
 i_RdP=mP,
\]
so $dP$ cannot vanish on $P^{-1}(1)$.

The polynomial $P$ defines the global Milnor fibration
\[
 P:\CC^3\setminus P^{-1}(0)\longrightarrow\CC^*.
\]
Let
\[
 \zeta=e^{2\pi i/m}.
\]
A positive generator of $\pi_1(\CC^*)$ is represented by
$\gamma(t)=e^{2\pi it}$, $0\le t\le1$.  By homogeneity,
$P(\lambda u)=\lambda^mP(u)$, so for $u\in F$ the path
\[
 u(t)=e^{2\pi it/m}u
\]
lifts $\gamma$.  Hence the geometric monodromy of the global Milnor
fiber is the linear map
\[
 h:F\longrightarrow F,\qquad h(u)=\zeta u.
\]
The linearity here is a consequence of the homogeneity of $P$.

For $\lambda\in\CC^*$ we write
\[
 H^1(F,\CC)_\lambda
 =
 \{\xi\in H^1(F,\CC):h^*\xi=\lambda\xi\}
\]
for the $\lambda$-eigenspace of the monodromy action on cohomology.
If $\alpha$ is a homogeneous polynomial $1$-form of weight $q$, then
\[
 h^*(\alpha|_F)=\zeta^q\alpha|_F.
\]
Thus, whenever $\alpha|_F$ is closed, its cohomology class belongs to
$H^1(F,\CC)_{\zeta^q}$.

We shall use the following classical monodromy vanishing. A theorem of
Zariski on cyclic multiple planes implies that, for an irreducible plane
curve, a non-trivial root of unity of prime-power order cannot occur as
a root of the first Alexander polynomial. Equivalently, if
$\lambda\ne1$ has prime-power order, then the corresponding monodromy
eigenspace on the global Milnor fiber vanishes:
\[
 H^1(F,\CC)_\lambda=0.
\]
This is precisely the consequence of Zariski's theorem used below; see
the modern formulation recalled by Artal Bartolo--Dimca
\cite{ArtalDimca}. In particular, when $m=p^s$ and $1\le q<m$,
\[
 \ord(\zeta^q)=\frac{m}{\gcd(m,q)}
\]
is a non-trivial prime power, and hence
\[
 H^1(F,\CC)_{\zeta^q}=0.
\]

For the trivial eigenvalue one has
\[
 H^1(F,\CC)_1\simeq H^1(U,\CC),
 \qquad U=\PP^2\setminus C.
\]
For a reduced irreducible plane curve of degree $m$,
\[
 H_1(U,\ZZ)\simeq\ZZ/m\ZZ,
\]
and therefore
\[
 H^1(U,\CC)=0.
\]
More generally, if a reduced plane curve has irreducible components of
degrees $d_1,\ldots,d_r$, its first homology is generated by meridians
$\mu_1,\ldots,\mu_r$ subject to the single relation
$\sum_i d_i\mu_i=0$; see \cite{ArtalDimca}.

Finally, we recall why algebraic de Rham theory is needed after the
cohomological vanishing.  The ordinary Poincar\'e lemma gives only
local primitives.  Here, once a closed algebraic $1$-form on the
smooth affine variety $F$ has zero class in $H^1(F,\CC)$,
Grothendieck's algebraic de Rham theorem identifies this vanishing with
exactness in the algebraic de Rham complex.  Consequently there exists
a \emph{global regular algebraic} function $g\in\OO(F)$ such that
\[
 \alpha|_F=dg.
\]
This algebraicity, rather than merely local holomorphic exactness, is
what will make the descent and denominator-elimination argument
possible; see \cite{HartshorneDR}.

\begin{lemma}[Twisted vanishing]\label{lem:vanishing}
Let $P\in\CC[x,y,z]$ be homogeneous, reduced and irreducible of degree
$m=p^s$. Let $\alpha$ be a homogeneous polynomial $1$-form of weight
$q$, $1\leq q\leq m$, satisfying
\[
 i_R\alpha=0,\qquad
 P\,d\alpha+\frac qm\alpha\wedge dP=0.
\]
Then $\alpha=0$.
\end{lemma}

\begin{proof}
Restrict the twisted equation to the Milnor fiber
\[
 F=\{P=1\}.
\]
Since $dP$ vanishes on tangent vectors to $F$, we obtain
\[
 d_F(\alpha|_F)=0.
\]
By homogeneity and the description of the monodromy above,
\[
 h^*(\alpha|_F)=\zeta^q\alpha|_F,
\]
and therefore
\[
 [\alpha|_F]\in H^1(F,\CC)_{\zeta^q}.
\]
If $q<m$, the eigenvalue $\zeta^q$ has non-trivial prime-power order,
so the classical vanishing recalled above gives
\[
 H^1(F,\CC)_{\zeta^q}=0.
\]
If $q=m$, then $\zeta^q=1$ and
\[
 H^1(F,\CC)_1\simeq H^1(\PP^2\setminus\{P=0\},\CC)=0.
\]
Thus in every case
\[
 [\alpha|_F]=0.
\]

Since $F$ is smooth affine, algebraic de Rham theory gives a global
regular function $g\in\OO(F)$ such that
\[
 \alpha|_F=dg.
\]
We now choose the primitive so that it transforms under the monodromy
in exactly the same way as $\alpha|_F$.  Define the
$\zeta^q$-isotypical projection of $g$ by
\[
 g_q(u)
 =
 \frac1m\sum_{\ell=0}^{m-1}
 \zeta^{-q\ell}g(\zeta^\ell u).
\]
Because
\[
 d(g\circ h^\ell)
 =h^{\ell *}(dg)
 =\zeta^{q\ell}dg,
\]
we have
\[
 dg_q
 =
 \frac1m\sum_{\ell=0}^{m-1}
 \zeta^{-q\ell}\zeta^{q\ell}dg
 =dg.
\]
Moreover, a change of index gives
\[
 g_q(\zeta u)=\zeta^qg_q(u).
\]
Replacing $g$ by $g_q$, we may therefore assume, without changing its
differential, that
\begin{equation}\label{eq:g-character}
 g(\zeta u)=\zeta^qg(u).
\end{equation}
For $q=m$ this is simply the invariant averaging of $g$ over the
cyclic group.

On
\[
 V=D(P)=\CC^3\setminus\{P=0\}
\]
consider the finite \'etale covering
\[
 \pi:\CC^*\times F\longrightarrow V,\qquad (r,u)\longmapsto ru.
\]
Since $P(u)=1$, homogeneity gives $P(ru)=r^m$.  The deck group is
\[
 \mu_m=\{\xi\in\CC^*:\xi^m=1\},
\]
the cyclic group of $m$-th roots of unity, acting by
\[
 \xi\cdot(r,u)=(\xi r,\xi^{-1}u).
\]
Equation \eqref{eq:g-character} holds with $\zeta$ replaced by every
$\xi\in\mu_m$, since each such $\xi$ is a power of $\zeta$.  Hence
\[
 g(\xi^{-1}u)=\xi^{-q}g(u),
\]
and therefore, under the deck transformation above,
\[
 (\xi r)^qg(\xi^{-1}u)
 =\xi^q r^q\,\xi^{-q}g(u)
 =r^qg(u).
\]
Thus
\[
 \widetilde f(r,u)=r^qg(u)
\]
is invariant under the deck group and descends to a single-valued
regular function
\[
 f\in\OO(V).
\]

For completeness, we record the differential identity satisfied by
$f$.  Homogeneity of $\alpha$ and the horizontality relation
$i_R\alpha=0$ imply
\[
 \pi^*\alpha=r^q\,\alpha|_F.
\]
Since $dg=\alpha|_F$,
\[
 d\widetilde f
 =q r^{q-1}g(u)\,dr+r^q\alpha|_F
 =\pi^*\alpha+\frac qm\widetilde f\,
 \pi^*\!\left(\frac{dP}{P}\right),
\]
because
\[
 \pi^*\!\left(\frac{dP}{P}\right)=m\,\frac{dr}{r}.
\]
Descending this equality to $V$ yields
\begin{equation}\label{eq:f}
 \alpha=df-\frac qm f\frac{dP}{P}.
\end{equation}

Since
\[
 \OO(V)=\CC[x,y,z,P^{-1}]
\]
and $f$ is homogeneous of weight $q$, we may write
\[
 f=\frac{G}{P^k},
\]
where $k\ge0$ is chosen minimal and $G$ is homogeneous of degree
$q+km$.  Minimality means that, if $k>0$, then $P\nmid G$.
Substituting this expression into \eqref{eq:f} and multiplying by
$P^{k+1}$ gives
\begin{equation}\label{eq:denominator}
 P^{k+1}\alpha
 =
 P\,dG-\left(k+\frac qm\right)G\,dP.
\end{equation}
Reducing modulo $P$, the left-hand side and the term $P\,dG$ vanish,
so
\[
 \left(k+\frac qm\right)G\,dP\equiv0\pmod P.
\]
Since $k+q/m\ne0$, this gives
\[
 G\,dP\equiv0\pmod P.
\]
The polynomial $P$ is irreducible and reduced, so $dP$ is not
identically zero along $P=0$; equivalently, at least one partial
derivative of $P$ is not divisible by $P$.  Hence the last congruence
implies
\[
 P\mid G.
\]
If $k>0$, this contradicts the minimality of $k$, since writing
$G=PG_1$ would give
\[
 f=\frac{G_1}{P^{k-1}}.
\]
Consequently
\[
 k=0.
\]

With $k=0$, equation \eqref{eq:denominator} still implies $P\mid G$.
If $q<m$, then $\deg G=q<m=\deg P$, hence $G=0$, and
\eqref{eq:f} gives $\alpha=0$.  If $q=m$, then
\[
 \deg G=\deg P=m.
\]
Since $P\mid G$, the quotient $G/P$ has degree zero, so
\[
 G=cP
\]
for some constant $c\in\CC$.  Substituting into \eqref{eq:f}, and
using $q/m=1$, gives
\[
 \alpha
 =d(cP)-cP\,\frac{dP}{P}
 =c\,dP-c\,dP
 =0.
\]
This completes the proof.
\end{proof}

\begin{remark}
Let us give a few hints about the use of the prime-power hypothesis $m=p^s$ in Lemma~\ref{lem:vanishing}. This hypothesis enters  precisely through the monodromy. The
order of $\zeta^q$ is
\[
 \frac{m}{\gcd(m,q)},
\]
and is therefore a prime power when $m=p^s$. If $m$ is not a prime
power, this order may have at least two distinct prime divisors, and the
corresponding eigenspace $H^1(F,\CC)_{\zeta^q}$ need not vanish. In
that case a non-zero twisted cohomology class may survive and provide a
genuine obstruction to the conclusion $\alpha=0$. Later on, we shall show through an example that  this obstruction phenomenon actually occurs (see Section~\ref{sec:sharpness}).
\end{remark}

\section{Proof of Theorem~\ref{thm:main}}

The last step before proving Theorem~\ref{thm:main} is the following  passage from the dimension $3$ case to arbitrary dimension.
\begin{lemma}[Linear sections]\label{lem:linear}
Let $\alpha$ be a homogeneous horizontal polynomial $1$-form on
$\CC^n$. If its restriction to a general three-dimensional linear
subspace is zero, then $\alpha=0$.
\end{lemma}

\begin{proof}
If $\alpha\neq0$, choose $x\neq0$ and $v$ with
$\alpha_x(v)\neq0$. Since $\alpha_x(x)=0$, choose $v$ independent of
$x$ and complete $x,v$ to a three-dimensional vector subspace $W$.
Then $\alpha|_W\neq0$. The condition $\alpha|_{W'}=0$ is algebraic
in $W'\in\Gr(3,n)$: in a local chart of the Grassmannian it is the
simultaneous vanishing of the polynomial coefficients of the
restricted form. Hence
\[
 \{W'\in\Gr(3,n):\alpha|_{W'}\neq0\}
\]
is a non-empty Zariski open subset containing $W$.
\end{proof}

\begin{proof}[Proof of Theorem~\ref{thm:main}]
If some $\alpha_j$ in \eqref{eq:difference} is non-zero, choose the
first one and put $q=m-j$. Since the coefficients of $\alpha_j$ are
homogeneous of degree $m-j-1$, the $1$-form $\alpha_j$ has weight
$q=m-j$. Lemma~\ref{lem:first} gives
\[
 i_{R_H}\alpha_j=0,\qquad
 P\,d\alpha_j+\frac qm\alpha_j\wedge dP=0.
\]
Choose a general plane $L\simeq\PP^2$ in $H$. When $\dim H>2$,
obtain $L$ by successive general hyperplane sections. Bertini's
irreducibility theorem then gives that $S\cap L$ is a reduced
irreducible plane curve of degree $m=p^s$.
The restriction $\alpha_j|_L$ satisfies Lemma~\ref{lem:vanishing};
hence it is zero for every $L$ in the non-empty Bertini open subset of
the Grassmannian. If $\alpha_j\neq0$, Lemma~\ref{lem:linear} and its
proof give a non-empty Zariski open subset on which
$\alpha_j|_L\neq0$. Since the Grassmannian is irreducible, these two
non-empty open subsets intersect, a contradiction. Thus
$\alpha_j=0$.
Therefore all $\alpha_j$ vanish and
\[
 \Omega=Q\,dt-\frac tm dQ
 =-\frac{t^{m+1}}m\,d\left(\frac{Q}{t^m}\right).
\]
\end{proof}

\section{Smooth invariant hypersurfaces: proof of Theorem~\ref{thm:smooth-hypersurface}}\label{sec:smooth-hypersurface}

We now prove Theorem~\ref{thm:smooth-hypersurface}. The argument
parallels the proof of Theorem~\ref{thm:main}. There, the remainder is
analyzed successively according to its order of divisibility by the
linear equation $t$ of the invariant hyperplane. Here we use the same
procedure with the defining polynomial $F$ of the invariant hypersurface
$D=(F=0)$; in this sense the argument is {\it $F$-adic}.  Zariski's
prime-power vanishing (see \S~\ref{subsection:zariskivanishing} and Lemma~\ref{lem:vanishing})  is replaced by the topology of the cyclic cover
branched along the smooth divisor $S$.

Our strategy is to prove a couple of similar vanishing statements  needed in the induction process similarly to the proof of Theorem~\ref{thm:main}.
The first is an  untwisted version, in which no smoothness
assumption on the divisor is needed.

\begin{lemma}[Untwisted vanishing for an irreducible divisor on a smooth hypersurface]
\label{lem:untwisted-irreducible}
Let $D=(F=0)\subset\PP^n$, $n\ge3$, be a smooth hypersurface of degree
$e$, and let $P$ be homogeneous of degree $r>0$ such that
\[
 S=D\cap(P=0)
\]
is reduced and irreducible. Let $\alpha$ be a homogeneous polynomial
$1$-form of weight $r$ satisfying $i_R\alpha=0$ and, on the affine cone
$\widehat D=(F=0)$,
\[
 P\,d\alpha+\alpha\wedge dP=0.
\]
Then the restriction of $\alpha$ to $D$ vanishes.
\end{lemma}

\begin{proof}
Put $U=D\setminus S$ and
\[
 M=\{F=0,\ P=1\}\subset\CC^{n+1}.
\]
The variety $M$ is smooth, by the same Euler-field argument used below,
and $M\to U$ is the unramified cyclic cover of degree $r$. Restriction
to $M$ gives a closed form $\alpha_M$. Since $\alpha$ has weight $r$,
its cohomology class is invariant under the deck group. Hence
\[
 [\alpha_M]\in H^1(M,\CC)_1\simeq H^1(U,\CC).
\]

We next show directly:
\begin{claim}
We have  $H^1(U,\CC)=0$.
\end{claim}
\begin{proof}[proof of the claim]
The claim is proved  without any smoothness
assumption on $S$. Since $D$ is a smooth hypersurface of dimension at
least two, the Lefschetz hyperplane theorem gives $\pi_1(D)=0$.
Let $\mu$ be a small positively oriented meridian around a smooth point
of $S$. Any loop in $U$ bounds a disk in $D$; after a small perturbation,
the disk can be chosen to avoid $\Sing(S)$ and to meet $S_{\rm reg}$
transversely in finitely many points. Removing small disks around these
intersection points expresses the original loop as a product of
conjugates of meridians or their inverses. Because $S$ is irreducible, all
such meridians are conjugate up to orientation. Hence $H_1(U,\ZZ)$ is
cyclic, generated by $[\mu]$.

Choose now a general complete-intersection curve $C\subset D$, obtained
by intersecting $D$ with $n-2$ general hyperplanes. We may assume that
$C$ avoids $\Sing(S)$ and meets $S$ transversely. Since
$S\in|\mathcal O_D(r)|$ and $\deg D=e$, one has
\[
 C\cdot S=re.
\]
Removing small disks from $C$ around the $re$ intersection points shows
that the oriented boundary of the resulting surface represents
$re[\mu]$. Therefore
\[
 re[\mu]=0\qquad\hbox{in }H_1(U,\ZZ).
\]
Thus $H_1(U,\ZZ)$ is finite cyclic, and consequently
\[
 H^1(U,\CC)=0.
\]
\end{proof}

Next we have:
\begin{claim}
We have $\alpha_M=dg$ for some regular function $g\in\mathcal O(M)$ which is invariant under
$\Gamma$.
\end{claim}
\begin{proof}
It follows from the claim above that $[\alpha_M]=0$.
Because $M$ is smooth affine, algebraic de Rham theory yields
$\alpha_M=dg$ for some regular $g\in\mathcal O(M)$. Since $dg$ is invariant
under the deck group $\Gamma$, for every deck transformation
$\tau\in\Gamma$ one has
\[
 d(\tau^*g-g)=0,
\]
so $\tau^*g-g$ is constant. Replacing $g$ by the average
\[
 \widetilde g=\frac1{|\Gamma|}\sum_{\tau\in\Gamma}\tau^*g
\]
gives another primitive, $d\widetilde g=dg$, which is invariant under
$\Gamma$.
\end{proof}

Consider the $r$-fold covering
\[
\CC^*\times M \longrightarrow V,
\qquad
(\rho,u)\longmapsto \rho u.
\]
Since $\widetilde g$ is invariant under the deck group and $\alpha$ has
weight $r$, the function
\[
\rho^r\widetilde g(u)
\]
is invariant under the deck action and therefore descends to a homogeneous
regular function $f$ of degree $r$ on $V$. A direct differentiation gives
\[
\alpha=df-f\frac{dP}{P}.
\]

Let
\[
 A_D=\CC[x_0,\ldots,x_n]/(F)
\]
be the homogeneous coordinate ring of the affine cone. Then
$\mathcal O(V)=A_D[P^{-1}]$, so we may write, with minimal $k\ge0$,
\[
 f=\frac{G}{P^k},\qquad G\in A_D
\]
with $G$ homogeneous of degree $(k+1)r$. Minimality means that $P\nmid G$
if $k>0$. Multiplying the identity above by $P^{k+1}$ gives
\[
 P^{k+1}\alpha=P\,dG-(k+1)G\,dP.
\]
Since $S$ is reduced and irreducible, the image of $P$ in $A_D$ is prime,
and $dP$ is non-zero at the generic point of $S$. Reducing modulo $P$
therefore forces $P\mid G$. This contradicts minimality when $k>0$, so
$k=0$. The same reduction then gives $P\mid G$; since $G$ and $P$ both
have degree $r$, we have $G=cP$ for some $c\in\CC$. Consequently
$\alpha=d(cP)-cP\,dP/P=0$ on $D$.
\end{proof}

The next vanishing statement need in the induction is the following:

\begin{lemma}[Twisted vanishing on a smooth hypersurface]
\label{lem:smooth-twisted}
Let $D=(F=0)\subset\PP^n$, $n\ge3$, be a smooth hypersurface of degree
$e$, and let $P$ be a homogeneous polynomial of degree $r>0$ such that
\[
 S=D\cap(P=0)
\]
is smooth and irreducible.  Let $\alpha$ be a homogeneous polynomial
$1$-form of weight $q$, $1\le q\le r$, satisfying
\[
 i_R\alpha=0
\]
and, after restriction to the affine cone $\widehat D=(F=0)$,
\begin{equation}\label{eq:smooth-twisted}
 P\,d\alpha+\frac qr\alpha\wedge dP=0.
\end{equation}
Then the restriction of $\alpha$ to $D$ vanishes as a differential
$1$-form on $D$.
\end{lemma}

We shall need a few claims and one remark before the proving the lemma.  
\begin{proof}[Beginning of the proof of Lemma~\ref{lem:smooth-twisted}]
Consider the affine complete intersection
\[
 M=\{F=0,\ P=1\}\subset\CC^{n+1}.
\]

\begin{claim}
$M$  is smooth.
\end{claim}
\begin{proof}[proof of the claim] Indeed, by the Jacobian criterion it is enough to show
that $dF$ and $dP$ are linearly independent at every point of $M$.
Suppose instead that $dP=\lambda dF$ at some $u\in M$. Since $D$ is
smooth and $u\neq0$, one has $dF(u)\neq0$. Let $R$ denote the Euler
field. By homogeneity,
\[
 dP(R)=rP=r, \qquad dF(R)=eF=0
\]
on $M$. Evaluating $dP=\lambda dF$ on $R$ would therefore give $r=0$,
a contradiction. Hence $M$ is smooth.
\end{proof}

Restricting \eqref{eq:smooth-twisted} to $M$ shows that
$\alpha_M:=\alpha|_M$ is closed.

A more delicate step is the following:
\begin{claim}
 We have 
\begin{equation}\label{eq:H1Mzero}
 H^1(M,\CC)=0.
\end{equation}
\end{claim}
\begin{proof}[proof of the claim]
To prove this, we realize $M$ as the complement of the ramification
divisor in a smooth projective cyclic cover of $D$. Let
\[
 L=\mathcal O_D(1).
\]
The homogeneous polynomial $P$ of degree $r$ restricts to a section
\[
 P|_D\in H^0(D,L^r),
\]
whose zero divisor is $S=D\cap(P=0)$. Let
\[
 \pi:X\longrightarrow D
\]
be the cyclic cover of degree $r$ associated with this section (for
$r=1$, simply take $X=D$). Since $S$ is smooth, $X$ is smooth
projective, and
\[
 \pi_*\mathcal O_X=\bigoplus_{i=0}^{r-1}L^{-i}.
\]
For every $i\ge0$ one has
\[
 H^1(D,\mathcal O_D(-i))=0.
\]
Indeed, this follows from
\[
 0\longrightarrow\mathcal O_{\PP^n}(-e-i)
 \longrightarrow\mathcal O_{\PP^n}(-i)
 \longrightarrow\mathcal O_D(-i)\longrightarrow0
\]
together with the standard vanishing
\[
 H^j(\PP^n,\mathcal O_{\PP^n}(k))=0,
 \qquad 0<j<n,\quad k\in\ZZ,
\]
for the intermediate cohomology of line bundles on projective space
(see \cite[Chapter~III, Theorem~5.1]{HartshorneAG}). Hence
$H^1(X,\mathcal O_X)=0$. Since $X$ is smooth and projective, the Hodge
decomposition theorem (see \cite[Section~6.1]{VoisinHodge}) gives
\[
 H^1(X,\CC)=H^{1,0}(X)\oplus H^{0,1}(X),
\]
while Dolbeault theory (see \cite[Section~2.3.3]{VoisinHodge}) identifies
$H^{0,1}(X)\simeq H^1(X,\mathcal O_X)$. Thus $H^{0,1}(X)=0$, and by
Hodge symmetry also $H^{1,0}(X)=0$. Consequently,
\[
 H^1(X,\CC)=0.
\]

Let $R_X\subset X$ be the ramification divisor of the cyclic cover
$\pi:X\to D$; when $r=1$ we put $R_X=S$. The restriction
$\pi|_{R_X}:R_X\to S$ is an isomorphism. Since $S$ is smooth and
irreducible, $R_X$ is connected, and therefore
\[
 H^0(R_X,\CC)\simeq\CC.
\]
For the smooth divisor $R_X\subset X$, the cohomology of the complement
fits into the standard Gysin long exact sequence, obtained from the
cohomology sequence of the pair $(X,X\setminus R_X)$ together with the
Thom isomorphism; see, for instance, \cite{BottTu}. The relevant segment
is
\[
 H^1(X,\CC)\longrightarrow H^1(X\setminus R_X,\CC)
 \longrightarrow H^0(R_X,\CC)
 \xrightarrow{\mathrm{Gys}} H^2(X,\CC).
\]
The Gysin map sends $1\in H^0(R_X,\CC)$ to the divisor class
$[R_X]\in H^2(X,\CC)$, equivalently to the Poincar\'e dual of $R_X$.
This class is non-zero. Indeed, if $h$ is the first Chern class of an
ample line bundle on $X$ and $N=\dim X$, then
\[
 \int_X [R_X]\smile h^{N-1}
 =\int_{R_X}h^{N-1}>0,
\]
because the restriction of an ample class to the non-zero effective
divisor $R_X$ is ample. Hence $[R_X]\neq0$. Since
$H^0(R_X,\CC)$ is one-dimensional, the Gysin map is injective.
Together with $H^1(X,\CC)=0$, exactness gives
\[
 H^1(X\setminus R_X,\CC)=0.
\]

It remains to identify this complement with $M$. If $u\in M$, then
$F(u)=0$ and $P(u)=1$, so its projective class $[u]$ belongs to
$D\setminus S$. Conversely, for each $[u]\in D\setminus S$, the
homogeneity relation $P(\lambda u)=\lambda^rP(u)$ shows that there are
exactly $r$ choices of scalar $\lambda$ for which $P(\lambda u)=1$.
They differ by the action of $\mu_r$. Thus
\[
 M\longrightarrow D\setminus S,\qquad u\longmapsto[u],
\]
is the unramified cyclic cover of degree $r$ obtained by restricting
$X\to D$ away from the branch divisor. Consequently
\[
 M\simeq X\setminus R_X.
\]
Therefore
\[
 H^1(M,\CC)=H^1(X\setminus R_X,\CC)=0,
\]
which proves \eqref{eq:H1Mzero}.
\end{proof}

We now finish the proof of Lemma~\ref{lem:smooth-twisted}. Since $M$ is smooth affine, algebraic de Rham theory now gives
\[
 \alpha_M=dg
\]
for some $g\in\mathcal O(M)$.  Let $\zeta$ be a primitive $r$-th root
of unity.  The deck transformation $u\mapsto\zeta u$ satisfies
$\zeta^*\alpha_M=\zeta^q\alpha_M$.  Replacing $g$ by its
$\zeta^q$-isotypical component (and subtracting a constant if $q=r$),
we may assume
\[
 g(\zeta u)=\zeta^q g(u).
\]

Put
\[
 V=\widehat D\cap\{P\ne0\}.
\]
The map
\[
 \CC^*\times M\longrightarrow V,\qquad (\rho,u)\longmapsto\rho u,
\]
is an $r$-fold cover.  Homogeneity and $i_R\alpha=0$ give
\[
 \alpha_{\rho u}=\rho^q\alpha_u
\]
in pull-back form.  Consequently $\rho^qg(u)$ descends to a homogeneous
regular function $f$ of degree $q$ on $V$, and
\begin{equation}\label{eq:smooth-descent}
 \alpha=df-\frac qr f\,\frac{dP}{P}
\end{equation}
on $V$, as a differential form on $\widehat D$.

Let
\[
 A_D=\CC[x_0,\ldots,x_n]/(F)
\]
be the homogeneous coordinate ring of the affine cone.  Since
$S$ is irreducible, the image of $P$ in $A_D$ is prime.  Write, with
minimal $k\ge0$,
\[
 f=\frac{G}{P^k},\qquad
 G\in A_D\ \hbox{homogeneous of degree }q+kr.
\]
Equation \eqref{eq:smooth-descent} gives in $\Omega^1_{A_D}$
\[
 P^{k+1}\alpha
 =P\,dG-\left(k+\frac qr\right)G\,dP.
\]
Reducing modulo $P$, and using that $S$ is reduced, we obtain $P\mid G$.
Indeed, $dP$ is non-zero at the generic point of $S$.  This contradicts
minimality if $k>0$, so $k=0$.  The same argument then gives $P\mid G$.
If $q<r$, homogeneity forces $G=0$; if $q=r$, then $G=cP$ and
\eqref{eq:smooth-descent} again gives $\alpha=0$ in
$\Omega^1_{A_D}$.  Thus the pull-back of $\alpha$ to $D$ vanishes.
\end{proof}

Now we are ready to prove our second main result:

\begin{proof}[Proof of Theorem~\ref{thm:smooth-hypersurface}]
A degree-$d$ foliation is represented by an integrable homogeneous
$1$-form whose coefficients have degree $d+1$.  Since $D=(F=0)$ is
invariant and smooth, the conormal sequence gives a decomposition
\begin{equation}\label{eq:Fadic-start}
 \Omega=A\,dF+F\beta,
\end{equation}
where $A$ is homogeneous of degree
\[
 r=d-e+2.
\]
The Euler condition yields
\begin{equation}\label{eq:Fadic-Euler}
 i_R\beta=-eA.
\end{equation}
Notice that the positivity of $r$ need not be assumed separately.
Indeed, the primitiveness of $\Omega$ forces $A\ne0$: if $A=0$, then
\eqref{eq:Fadic-start} would give $\Omega=F\beta$, so all coefficients
of $\Omega$ would have the non-constant common factor $F$.  Hence
$r\ge0$.  If $r=0$, then $A$ is a non-zero constant and, along the
smooth hypersurface $D$,
\[
 \Omega|_D=A\,dF.
\]
Since $dF$ does not vanish on the smooth locus of the affine cone over
$D$ away from the origin, this would imply
$D\cap\Sing(\F)=\varnothing$, contradicting the hypothesis that $S$ is
a hypersurface of $D$.  Therefore $r>0$ follows automatically.

The divisor of $A|_D$ belongs to $|\mathcal O_D(r)|$ and therefore has
projective degree $er$.  By the hypotheses on $S$, this divisor is
exactly $S$ with multiplicity one.

For a homogeneous polynomial $Q$ of degree $r$, set
\[
 \Omega_Q=Q\,dF-\frac erF\,dQ.
\]
It is integrable and
\[
 \frac{\Omega_Q}{FQ}
 =\frac{dF}{F}-\frac er\frac{dQ}{Q}
 =-\frac1r\,d\log\!\left(\frac{Q^e}{F^r}\right),
\]
so $Q^e/F^r$ is a rational first integral.

Starting with $Q=A$, equations \eqref{eq:Fadic-start} and
\eqref{eq:Fadic-Euler} give
\[
 \Omega=\Omega_Q+F\alpha_0,
 \qquad
 \alpha_0=\beta+\frac er dA,
 \qquad
 i_R\alpha_0=0.
\]
The form $\alpha_0$ has weight $r$.

Suppose inductively that
\begin{equation}\label{eq:Fadic-induction}
 \Omega=\Omega_Q+F^{j+1}\alpha_j,
 \qquad
 i_R\alpha_j=0,
\end{equation}
where $Q$ has degree $r$ and $\alpha_j$ has weight
\[
 q_j=r-je>0.
\]
Taking the first non-zero coefficient of
$\Omega\wedge d\Omega=0$ along $F=0$ gives
\begin{equation}\label{eq:Fadic-twisted}
 Q\,d\alpha_j+\frac{q_j}{r}\alpha_j\wedge dQ=0
 \qquad\hbox{on }\widehat D.
\end{equation}
Indeed, writing $\delta=F^{j+1}\alpha_j$ and using
\[
 d\Omega_Q=\left(1+\frac er\right)dQ\wedge dF,
\]
the coefficient of $F^{j+1}$ in
$\Omega_Q\wedge d\delta+\delta\wedge d\Omega_Q$ is
\[
 dF\wedge\left(
 Q\,d\alpha_j+\frac{r-je}{r}\alpha_j\wedge dQ
 \right).
\]
Restriction to $D$ gives \eqref{eq:Fadic-twisted}.

The divisor $(Q=0)$ on $D$ is still $S$: all changes of $Q$ below are
multiples of $F$, hence do not change $Q|_D$.  Lemma~\ref{lem:smooth-twisted}
applied with $P=Q$ therefore shows that the restriction of
$\alpha_j$ to $D$ vanishes.  Consequently
\begin{equation}\label{eq:Fadic-divide}
 \alpha_j=h_j\,dF+F\beta_j
\end{equation}
for a homogeneous polynomial $h_j$ of degree $q_j-e$ (with $h_j=0$ if
$q_j<e$).  Contracting with the Euler field gives
\[
 i_R\beta_j=-e h_j.
\]

If $q_j<e$, then $h_j=0$ and $\beta_j$ would have non-positive weight,
so $\alpha_j=0$.  If $q_j=e$, then $h_j$ is constant and $\beta_j$ has
weight zero, hence $\beta_j=0$; horizontality then forces $h_j=0$.
Thus again $\alpha_j=0$.

Assume finally that $q_j>e$.  Put
\[
 B_j=\frac{r}{q_j-e}F^{j+1}h_j,
\]
which is homogeneous of degree $r$.  A direct computation gives
\[
 \Omega_{B_j}
 =F^{j+1}h_j\,dF
 -\frac{e}{q_j-e}F^{j+2}dh_j.
\]
Replacing $Q$ by $Q+B_j$ in \eqref{eq:Fadic-induction} yields
\[
 \Omega=\Omega_{Q+B_j}+F^{j+2}\alpha_{j+1},
\]
where
\[
 \alpha_{j+1}=\beta_j+\frac{e}{q_j-e}dh_j.
\]
This form is horizontal and has weight
\[
 q_{j+1}=q_j-e=r-(j+1)e.
\]
Thus the induction continues.  Since the weights decrease by $e$, it
terminates after finitely many steps, and the final remainder is zero.
Hence
\[
 \Omega=Q\,dF-\frac erF\,dQ
\]
for a homogeneous polynomial $Q$ of degree $r$, as claimed.
\end{proof}

\begin{proof}[Proof of Corollary~\ref{cor:quadric-cubic}]
The $F$-adic construction in the proof of
Theorem~\ref{thm:smooth-hypersurface} starts with a horizontal
remainder $\alpha_0$ of weight
\[
 q_0=r
\]
and, after each non-zero step, decreases the weight by $e$:
\[
 q_j=r-je.
\]
If $r\le e$, then $q_1=r-e\le0$, so no non-trivial twisted-character
step can occur. The only vanishing statement needed is the initial one
with $q_0=r$, and Lemma~\ref{lem:untwisted-irreducible} applies because
$S$ is reduced and irreducible; smoothness of $S$ is unnecessary.
The remainder of the $F$-adic argument is unchanged and gives
\[
 \Omega=Q\,dF-\frac erF\,dQ.
\]
For a quadric hypersurface, $e=2$ and $r=d$, so $0<r\le2$ means
$d=1,2$. For a cubic hypersurface, $e=3$ and $r=d-1$, so
$0<r\le3$ means $d=2,3,4$.
\end{proof}

\section{A few counterexamples on the hypotheses}\label{sec:sharpness}

In this section we show that the main hypotheses of
Theorem~\ref{thm:main} cannot be dropped. We first construct a non-reduced
example showing that the maximal-degree condition on the reduced
singular support cannot be replaced by a purely set-theoretic
assumption. We then give global counterexamples showing that both the
prime-power condition and the irreducibility of the reduced singular
support are essential.

\subsection{The $\deg S=d+1$ maximal-degree condition: a non-reduced example}
The following family shows what can occur once the maximal-degree
hypothesis
\[
\deg\bigl(H\cap\Sing(\F)\bigr)_{\rm red}=d+1
\]
is dropped. In the family below the reduced singular support remains
irreducible, but has degree strictly smaller than $d+1$.
We begin with a family whose reduced singular support on the invariant
hyperplane is irreducible and has finite non-trivial meridional
holonomy, but whose transverse singular scheme is non-reduced and which
admits no non-constant rational first integral.

Let $P,G$ be homogeneous of the same degree $m\geq2$, with $P$
irreducible and $P\nmid G$. Define
\begin{equation}\label{eq:family}
 \Omega_{P,G}=P^2dt+
 t\left(P\,dG-G\,dP-\frac1mP\,dP\right).
\end{equation}

\begin{proposition}\label{prop:family}
The 1-form \eqref{eq:family} is homogeneous, projective, primitive and
integrable. Its coefficients have degree $2m$, so it defines a
codimension-one foliation of degree
\[
 2m-1
\]
on $\PP^n$. Moreover,
\[
 \frac{\Omega_{P,G}}{tP^2}
 =\frac{dt}{t}-\frac1m\frac{dP}{P}
 +d\left(\frac GP\right).
\]
A positively oriented meridian around a smooth point of $\{P=0\}$
has transverse holonomy
\[
 \tau\longmapsto e^{2\pi i/m}\tau,
\]
up to the opposite convention for the orientation of the meridian.
Moreover, $\mathcal F_{P,G}$ admits no non-constant rational first
integral.
\end{proposition}

\begin{proof}
Since $P$ and $G$ are homogeneous of degree $m$, every coefficient of
$\Omega_{P,G}$ is homogeneous of degree $2m$. Euler's identity gives
\[
 i_R(P\,dG-G\,dP)=0,\qquad
 i_R\left(-\frac1mP\,dP\right)=-P^2,
\]
which proves $i_R\Omega_{P,G}=0$.

The form is primitive. Indeed, any common non-constant factor of its
coefficients must divide the coefficient $P^2$ of $dt$. Since $P$ is
irreducible, such a factor would be divisible by $P$. However, modulo
$P$ the tangential part satisfies
\[
 P\,dG-G\,dP-\frac1mP\,dP\equiv-G\,dP\pmod P.
\]
Because $P\nmid G$ and, in characteristic zero, $P$ does not divide
all coefficients of $dP$, the tangential coefficients are not all
divisible by $P$. Hence $\Omega_{P,G}$ is primitive and defines a
degree-$(2m-1)$ foliation.

Furthermore,
\[
 \frac{P\,dG-G\,dP}{P^2}=d(G/P),
\]
so the displayed rational $1$-form is closed and
$\Omega_{P,G}$ is integrable.

Finally, along $H\setminus\{P=0\}$ the transverse equation is
\[
 \frac{dt}{t}
 =
 \frac1m\frac{dP}{P}-d(G/P).
\]
If $\gamma$ is a positively oriented small meridian around a smooth
point of $\{P=0\}$, then
\[
 \int_\gamma\frac{dP}{P}=2\pi i,
 \qquad
 \int_\gamma d(G/P)=0.
\]
Therefore the holonomy germ in a transverse coordinate $\tau$ is
exactly
\[
 \tau\longmapsto e^{2\pi i/m}\tau.
\]

Finally, multiplying the closed rational form above by $m$ and
integrating gives the Liouvillian first integral
\[
I=\frac{t^m}{P}\exp\left(\frac{mG}{P}\right).
\]
Since $P\nmid G$, at a general smooth point of $(P=0)$ one has $G\ne0$.
Taking $P$ as a local transverse coordinate, the equation of a leaf
$I=C$ can be written
\[
t^m
=
C\,P\,\exp\left(-\frac{mG}{P}\right).
\]
The right-hand side has an essential singularity along $P=0$. More
precisely, after normalization an algebraic branch meeting a general
smooth point of $(P=0)$ would admit a Puiseux expansion in a local
parameter, whereas the exponential term above has an essential
singularity because $G$ does not vanish generically on $(P=0)$. Hence
a generic leaf is not algebraic. Since the generic leaves of a foliation
admitting a non-constant rational first integral are algebraic,
$\mathcal F_{P,G}$ admits no non-constant rational first integral.
\end{proof}

\subsection{Sharpness of the prime-power assumption}

The arithmetic hypothesis $m=d+1=p^s$ in Theorem~\ref{thm:main} is sharp. The construction below is the
global counterpart of the optimality example of Cerveau--Loray
\cite[Section~3]{CerveauLoray}.

\begin{proposition}[Sharpness of the prime-power hypothesis]\label{prop:primepower-counter}
Let $m\ge2$ be an integer which is not a power of a prime. Then there
exists a degree-$(m-1)$ codimension-one foliation $\mathcal G$ on
$\PP^3$ admitting an invariant hyperplane $H$ such that
\[
 S=\bigl(H\cap\Sing(\mathcal G)\bigr)_{\rm red}
\]
is irreducible of degree
\[
 \deg S=m=d+1,
\]
but $\mathcal G$ has no non-constant rational first integral. In
particular, the prime-power hypothesis in Theorem~\ref{thm:main} cannot
be removed.
\end{proposition}

\begin{proof}
Since $m$ is not a prime power, we can write
\[
 m=ab,
 \qquad a,b>1,
 \qquad \gcd(a,b)=1.
\]
Choose sufficiently general homogeneous polynomials
$A,B\in\CC[x,y,z]$ of degrees $a$ and $b$, respectively, such that
\[
 P=A^b+B^a
\]
is irreducible; this is the same torus-type construction used by
Cerveau--Loray in their optimality argument. Put
\[
 \alpha=bB\,dA-aA\,dB.
\]
The form $\alpha$ is homogeneous of weight
\[
 q=a+b<m=ab
\]
and is horizontal. Indeed, with $R_H$ the Euler field in the
$(x,y,z)$-variables,
\[
 i_{R_H}\alpha
 =bB(aA)-aA(bB)=0.
\]
Moreover,
\[
 d\alpha=-(a+b)\,dA\wedge dB,
\]
while
\[
 dP=bA^{b-1}dA+aB^{a-1}dB.
\]
Hence
\[
 \alpha\wedge dP
 =ab\,(A^b+B^a)\,dA\wedge dB
 =mP\,dA\wedge dB,
\]
and therefore
\begin{equation}\label{eq:primepower-twisted}
 P\,d\alpha+\frac{a+b}{ab}\,\alpha\wedge dP=0.
\end{equation}
Thus $\alpha$ is a non-zero solution of the same twisted equation which
appears in the proof of Theorem~\ref{thm:main}.

Set
\[
 k=m-a-b+1=(a-1)(b-1)\ge2
\]
and, for $c\in\CC^*$, define on $\CC^4$ with homogeneous coordinates
$(t,x,y,z)$
\begin{equation}\label{eq:primepower-family}
 \Omega_c
 =P\,dt-\frac{t}{m}\,dP+c\,t^k\alpha.
\end{equation}
All coefficients of $\Omega_c$ are homogeneous of degree $m$ and
$i_R\Omega_c=0$, so the form is projective. It is primitive as well. In
fact, any common factor would divide the irreducible polynomial $P$, the
coefficient of $dt$; but the coefficient of $t$ in the tangential part
is $-(1/m)dP$, and a reduced irreducible polynomial cannot divide all
coefficients of its differential.

We next verify integrability. Write
\[
 \Omega_0=P\,dt-\frac{t}{m}dP,
 \qquad
 \delta=c\,t^k\alpha.
\]
Since $\Omega_0$ is integrable and $\alpha\wedge d\alpha=0$, a direct
calculation gives
\[
 \Omega_c\wedge d\Omega_c
 =c\,t^k dt\wedge
 \left(
 P\,d\alpha+\frac{a+b}{ab}\alpha\wedge dP
 \right),
\]
which vanishes by \eqref{eq:primepower-twisted}. Hence $\Omega_c$
defines a degree-$(m-1)$ foliation $\mathcal G_c$ on $\PP^3$.

The hyperplane $H=(t=0)$ is invariant and
\[
 \Omega_c|_H=P\,dt.
\]
Consequently
\[
 \bigl(H\cap\Sing(\mathcal G_c)\bigr)_{\rm red}=(P=0),
\]
which is irreducible of degree $m=d+1$.

It remains to prove that $\mathcal G_c$ has no non-constant rational
first integral. Let
\[
 F=\{P=1\}\subset\CC^3.
\]
The restriction $\alpha_F:=\alpha|_F$ is closed by
\eqref{eq:primepower-twisted}. We claim that
\[
 [\alpha_F]\neq0\quad\hbox{in }H^1(F,\CC).
\]
Indeed, if $[\alpha_F]=0$, then the descent and minimal-denominator
argument used in the proof of Lemma~\ref{lem:vanishing} applies verbatim
without any prime-power assumption and forces $\alpha=0$, a
contradiction. Thus there exists a closed loop $\gamma\subset F$ such
that
\[
 \int_\gamma\alpha_F\neq0.
\]

The map
\[
 \Phi:\CC\times F\longrightarrow\PP^3,
 \qquad
 (t,u)\longmapsto[t:u],
\]
is a finite unramified covering of the open set
\[
 V=\{[t:X]\in\PP^3:P(X)\neq0\},
\]
which contains $H\setminus(P=0)$. The pull-back foliation is defined by
\[
 \Phi^*\Omega_c=dt+c\,t^k\alpha_F.
\]
Along the lifted leaf $\{t=0\}\times F$ the holonomy equation is
therefore
\[
 dt+c\,t^k\alpha_F=0.
\]
Since $k\ge2$, setting $r=k-1$ gives
\[
 d(t^{-r})=r c\,\alpha_F.
\]
Hence the holonomy germ $h_\gamma$ associated with $\gamma$ satisfies
\[
 h_\gamma(t)^{-r}
 =t^{-r}+r c\int_\gamma\alpha_F.
\]
The period is non-zero, so $h_\gamma$ is a non-trivial germ tangent to
the identity and has infinite order. Therefore the holonomy group of
the invariant leaf $H\setminus S$ is infinite.

Finally, a rational first integral would force this holonomy group to be
finite. Indeed, at a generic point of $H$ one can choose a transverse
coordinate $\tau$ and a local coordinate on the target so that a
rational first integral has the form
\[
 f-f(H)=\tau^\nu u(\tau),
 \qquad u(0)\neq0.
\]
After replacing $\tau$ by $\tau u(\tau)^{1/\nu}$, every holonomy germ
preserves $\tau^\nu$, hence is multiplication by a $\nu$-th root of
unity. This contradicts the infinite-order germ above. Thus
$\mathcal G_c$ has no non-constant rational first integral.
\end{proof}

\subsection{Sharpness of irreducibility}

The following global example shows that irreducibility is genuinely
necessary in Theorem~\ref{thm:main} even when the maximal-degree condition and the prime-power
condition are both satisfied.

Take $m=2$, $P=xy$, and
\[
 \alpha=y\,dx-x\,dy.
\]
Then
\[
 i_R\alpha=0,\qquad
 P\,d\alpha+\alpha\wedge dP=0.
\]
For $c\in\CC$ define
\begin{equation}\label{eq:reducible-family}
 \Omega_c
 =
 P\,dt-\frac t2\,dP+c\,t\alpha.
\end{equation}
Explicitly,
\[
 \Omega_c
 =
 xy\,dt
 +
 \left(c-\frac12\right)ty\,dx
 -
 \left(c+\frac12\right)tx\,dy.
\]

\begin{proposition}\label{prop:reducible-counter}
Assume $c\neq\pm\frac12$. Then $\Omega_c$ is homogeneous, projective,
primitive and integrable, and defines a degree-one foliation
$\F_c$ on $\PP^3$. The hyperplane $H=(t=0)$ is invariant and
\[
 \bigl(H\cap\Sing(\F_c)\bigr)_{\rm red}
 =
 \{xy=0\},
\]
which is a reduced reducible conic of degree
\[
 2=d+1=2^1.
\]
Moreover,
\[
 \F_c\text{ admits a non-constant rational first integral}
 \quad\Longleftrightarrow\quad c\in\QQ.
\]
\end{proposition}

\begin{proof}
All coefficients of $\Omega_c$ are homogeneous of degree two. Moreover,
\[
 i_R\Omega_c
 =
 txy
 +
 \left(c-\frac12\right)txy
 -
 \left(c+\frac12\right)txy
 =0,
\]
so $\Omega_c$ is projective.

If $c\neq\pm\frac12$, the coefficients
\[
 xy,\qquad
 \left(c-\frac12\right)ty,\qquad
 -\left(c+\frac12\right)tx
\]
have no common non-constant factor. Hence $\Omega_c$ is primitive and
defines a degree-one foliation.

Dividing by $txy$, we obtain
\[
 \frac{\Omega_c}{txy}
 =
 \frac{dt}{t}
 +
 \left(c-\frac12\right)\frac{dx}{x}
 -
 \left(c+\frac12\right)\frac{dy}{y},
\]
which is closed. Thus $\Omega_c$ is integrable. On $H$ we have
$\Omega_c|_H=xy\,dt$, so the reduced singular support on $H$ is
$\{xy=0\}$. In fact, for $c\neq\pm1/2$ one has
\[
 \Sing(\F_c)
 =\{t=x=0\}\cup\{t=y=0\}\cup\{x=y=0\},
\]
so the singular set has pure codimension two.

It remains to characterize rational integrability. Put
\[
 a=c-\frac12,\qquad b=-c-\frac12.
\]
On the algebraic torus the defining closed form is
\[
 \theta=
 \frac{dt}{t}+a\frac{dx}{x}+b\frac{dy}{y}.
\]
Two commuting tangent vector fields are
\[
 D_1=x\frac{\partial}{\partial x}
     -a\,t\frac{\partial}{\partial t},
 \qquad
 D_2=y\frac{\partial}{\partial y}
     -b\,t\frac{\partial}{\partial t}.
\]
Hence any rational first integral belongs to the common constant field
of $D_1$ and $D_2$ on $\CC(t,x,y)$.

Assume first that $c\notin\QQ$. Consider the Laurent polynomial ring
\[
 A=\CC[t^{\pm1},x^{\pm1},y^{\pm1}].
\]
A monomial $M=t^r x^s y^u$ has joint weight
\[
 \bigl(s-ar,\;u-br\bigr).
\]
If two Laurent monomials have the same joint weight, their exponent
difference $(r,s,u)\in\ZZ^3$ satisfies
\[
 s=ar,\qquad u=br.
\]
Since $a=c-1/2\notin\QQ$, the case $r\neq0$ is impossible. Hence
$r=s=u=0$, and every joint weight space of $A$ is one-dimensional.

We now justify carefully the semi-invariant step. Let
$R=F/G\in\CC(t,x,y)$ be invariant under both $D_1,D_2$, with
coprime $F,G\in A$. For $i=1,2$,
\[
 (D_iF)G-F(D_iG)=0.
\]
Since $A$ is a UFD and $F,G$ are coprime,
\[
 F\mid D_iF,\qquad G\mid D_iG.
\]
Thus $D_iF=h_iF$ for some $h_i\in A$. Because $D_i$ acts diagonally
on Laurent monomials,
\[
 \operatorname{Supp}(D_iF)\subseteq\operatorname{Supp}(F).
\]
On the other hand, if $D_iF=h_iF\neq0$, the Newton polytopes satisfy
\[
 \operatorname{Newt}(h_i)+\operatorname{Newt}(F)
 =\operatorname{Newt}(h_iF)
 \subseteq\operatorname{Newt}(F).
\]
Since $\operatorname{Newt}(F)$ is bounded, this forces
$\operatorname{Newt}(h_i)=\{0\}$, hence $h_i\in\CC$. The same
argument applies to $G$, and the identity above shows that $F$ and $G$
have the same eigenvalue for each $D_i$. Therefore they are simultaneous
semi-invariants with the same joint weight. Since every joint weight
space is one-dimensional, $F$ and $G$ are proportional Laurent
monomials, and $R\in\CC$.

Conversely, suppose that $c\in\QQ$. Choose a positive integer $N$
such that
\[
 N\left(c-\frac12\right),\qquad
 N\left(c+\frac12\right)
\]
are integers. Then
\[
 I=t^N x^{N(c-1/2)}y^{-N(c+1/2)}
\]
is a non-constant rational function and
\[
 \frac{dI}{I}=N\theta.
\]
Thus $I$ is a rational first integral. This proves the equivalence.
\end{proof}

\end{document}